\documentclass[10pt, a4paper]{amsart}
\pdfoutput=1
\usepackage{amsmath,amsfonts,amssymb,amsthm,xcolor,hyperref,enumitem}  
\usepackage{mathrsfs} 

\usepackage[utf8]{inputenc}
\usepackage{graphicx}

\def\Xint#1{\mathchoice
{\XXint\displaystyle\textstyle{#1}}%
{\XXint\textstyle\scriptstyle{#1}}%
{\XXint\scriptstyle\scriptscriptstyle{#1}}%
{\XXint\scriptscriptstyle%
\scriptscriptstyle{#1}}%
\!\int}
\def\XXint#1#2#3{{\setbox0=\hbox{$#1{#2#3}{%
\int}$ }
\vcenter{\hbox{$#2#3$ }}\kern-.6\wd0}}
\def\barint{\,\Xint -} 
\def\bariint{\barint_{} \kern-.4em \barint}
\def\bariiint{\bariint_{} \kern-.4em \barint}
\renewcommand{\iint}{\int_{}\kern-.34em \int} 
\renewcommand{\iiint}{\iint_{}\kern-.34em \int} 

\theoremstyle{plain}
\newtheorem{theorem}{Theorem}[section]
\newtheorem{corollary}[theorem]{Corollary}
\newtheorem{lemma}[theorem]{Lemma}
\newtheorem{proposition}[theorem]{Proposition}

\theoremstyle{definition}

\newtheorem{remark}[theorem]{Remark}

\newtheorem{question}[theorem]{Question}

\numberwithin{equation}{section}
\newtheorem*{theorem*}{Theorem}

\newcommand{\R}{{\mathbb R}}

\newcommand{\sint}{\barint}

\DeclareMathOperator{\II}{II}
\DeclareMathOperator{\I}{I}
\DeclareMathOperator{\dive}{div}
\DeclareMathOperator{\dist}{dist}
\DeclareMathOperator{\diam}{diam}

\providecommand{\m}{M^\Omega}

\author{Jo\~ao P. G. Ramos, Olli Saari and Julian Weigt}
	
\address{Jo\~ao Pedro Ramos, Mathematical Institute, 
	University of Bonn,
	Endenicher Allee 60, 53115, Bonn,
	Germany}
	\email{jpgramos@math.uni-bonn.de}
	
\address{Olli Saari, Mathematical Institute, 
	University of Bonn,
	Endenicher Allee 60, 53115, Bonn,
	Germany}
	\email{saari@math.uni-bonn.de}
	
\address{Julian Weigt, Department of Mathematics and Systems Analysis, Aalto University, Finland}
	\email{julian.weigt@aalto.fi}
	


\date{\today}
\title[Local fractional maximal function]{Weak differentiability for fractional maximal functions of general $L^{p}$ functions on domains}

\begin{document}
\begin{abstract}
Let $\Omega \subset \mathbb{R}^{n}$ be bounded a domain. We prove under certain structural assumptions that the fractional maximal operator relative to $\Omega$ maps $L^{p}(\Omega) \to W^{1,p}(\Omega)$ for all $p > 1$, when the smoothness index $\alpha \geq 1$. In particular, the results are valid in the range $p \in (1, n/(n-1)]$ that was previously unknown. As an application, we prove an endpoint regularity result in the domain setting.
\end{abstract}

\maketitle

\section{Introduction}
Regularity of the Hardy--Littlewood maximal function of a Sobolev function was first studied in \cite{Kinnunen1997}. It was shown that the maximal operator preserves $W^{1,p}(\mathbb{R}^{n})$ regularity for $p>1$. This continues to hold true at the derivative level when $p=1$ and $n=1$ \cite{Tanaka2002,Kurka2010} and for radial functions \cite{Luiro2018}. Extending such a statement to more general Sobolev functions of several variables is a difficult open problem, which has inspired many results in related topics. For instance, slightly stronger bounds have been proved for maximal operators with more special convolution kernels (see \cite{CS2013}, \cite{CFS2018}, \cite{CR2019} and \cite{PPSS2018}), the continuity of the mapping has been studied in \cite{Luiro2007} and \cite{CMP2017}, and a part of the techniques used for continuity, also relevant for the current paper, have been extended to $p=1$ in \cite{HM2010}.

Another aspect of the problem is the fractional endpoint question proposed by Carneiro and Madrid \cite{CM2015}. The fractional maximal function is given by
\[M_\alpha f(x) = \sup_{r> 0} \frac{r^{\alpha}}{|B(x,r)|} \int_{B(x,r)} |f(y)| \, dy, \]
and it defines a bounded operator $L^{p}(\mathbb{R}^{n}) \to L^{q}(\mathbb{R}^{n})$ when $q = np/(n-p)$ and $p > 1$. This boundedness fails at the endpoint $p=1$, but the question about boundedness of $\nabla M_\alpha $ from  $W^{1,1}(\mathbb{R}^{n})$ to $L^{n/(n-\alpha)}(\mathbb{R}^{n})$ has not been answered so far for $\alpha < 1$ (see \cite{LM2017}, \cite{BM2019} and \cite{BRS2019} for related research and partial results). The case $\alpha \geq 1$ turned out to be very simple, and the reason can be traced back to the inequality
\begin{equation}
\label{eq:KSineq}
|\nabla M_\alpha f(x)| \leq c_{\alpha,n} M_{\alpha - 1} f(x)
\end{equation}
of Kinnunen and Saksman \cite{KS2003}. Carneiro and Madrid \cite{CM2015} noted that \eqref{eq:KSineq} together with the Gagliardo--Sobolev--Nirenberg inequality and the $L^{p} \to L^{q}$ bounds for the fractional maximal function imply the expected endpoint bound when $\alpha \geq 1$.  

In the present paper, we study these problems in general open subsets of $\mathbb{R}^{n}$, which is a natural context for analysis from the point of view of potential theory and partial differential equations. Regularity of the local Hardy--Littlewood maximal function of a Sobolev function on an open $\Omega \subset \mathbb{R}^{n}$ was first studied by Kinnunen and Lindqvist \cite{KL1998}, and a local variant of the inequality for the derivative of the fractional maximal function \eqref{eq:KSineq} was proved in \cite{HKKT2015}. This is our starting point, and for more thorough discussion of what was proved and what is unknown, we introduce some more notation.

If $\Omega \subset \mathbb{R}^{n}$ is an open set, the local fractional maximal function is defined as 
\[\m_{\alpha}f(x) = \sup_{0 < r < \dist(x, \Omega^c) } \frac{r^{\alpha}}{|B(x,r)|} \int_{B(x,r)} f(y) \, dy.\]
As the boundary of $\Omega$ restricts the choice of $r$ in the definition, one cannot expect \eqref{eq:KSineq} to trivially carry over to the local setting. Indeed, such a pointwise inequality is false in general (Example 4.1 in \cite{HKKT2015}). On the other hand, if one adds a correction term involving the surface measure of the sphere to the right hand side of \eqref{eq:KSineq}, one obtains
\begin{equation}
\label{eq:introHKKT}
|\nabla \m_\alpha f(x)| \leq c_{\alpha,n} \left( M_{\alpha - 1} f(x) +   \sup_{r > 0} |r^{\alpha-1} \sigma_r * f(x)|\right) ,
\end{equation}
which is valid in all domains. This was used in \cite{HKKT2015} to prove that $L^{p}$ functions with $p > n/(n-1)$ large enough have $\m_\alpha f$ in a first order Sobolev class. The lower bound on $p$ rules out functions too singular for an application of a spherical maximal function argument. 

Our main theorem shows that under suitable assumptions on the domain $\Omega$, the maximal function $\m_\alpha$ maps $L^{p}(\Omega)$ into a first order Sobolev space for all $p>1$.
\begin{theorem}
\label{thm:intro}
Let $\Omega \subset \mathbb{R}^{n}$ be open, $n \geq  2$, $p> 1$ and $f \in L^{p}(\Omega)$. Then $\m _\alpha f $ is weakly differentiable and 
\[\| \nabla \m _\alpha f \|_{L^{p}(\Omega)} \leq C  \| f \|_{L^{p}(\Omega)} \]
if any one of the following holds:
	\begin{enumerate}[label={(\arabic*)}]
	\item $\alpha > 1$ and $\Omega$ is bounded.\label{it:alphag1}
	\item $\alpha = 1$ and $\Omega^c$ is convex.\label{it:omegacc}
	\item $\alpha = 1$ and $\Omega$ is bounded and satisfies a uniform curvature bound in the sense of Section \ref{sec:dom}.\label{it:omegaib}
	\item $\alpha = 1$ and $p > 1 + \frac{1}{n}$.\label{it:pg11n}
\end{enumerate}
	The constant $C$ depends on the dimension, and in \ref{it:alphag1} and \ref{it:omegaib} it also depends on $\alpha$ and the domain.
\end{theorem} 
Unlike \cite{HKKT2015}, we are not able to prove an $L^{p} \to L^{q}$ smoothing effect on top of winning one derivative. However, our method does apply to singular functions in $L^{p}$ spaces with $1 \leq p \leq n/(n-1)$ where the argument in \cite{HKKT2015} fails to give any result. In particular, we have the following endpoint regularity result, which was previously out of reach. 

\begin{corollary}
\label{cor:intro2}
Let $\Omega \subset \mathbb{R}^{n}$ be a Lipschitz domain. Then for all $f \in W^{1,1}(\Omega)$
\[ \| \nabla \m _1 f \|_{L^{n/(n-1)}(\Omega)} \leq C \|  f \|_{W^{1,1}(\Omega)}  \]
where the constant $C$ only depends on $\Omega$ and the dimension. 
\end{corollary}

We briefly outline the proof of the main theorem. The maximal function on a domain behaves differently depending on whether the ball attaining the maximum touches the boundary or not. In case it does not, the local maximal function behaves like the global one, and the analysis is very similar. Otherwise it coincides with a linear averaging operator \eqref{def:solav}, which depends on the domain. These two parts are analyzed separately, and the main part of the proof is to establish $L^{p}$ bounds for the derivative of \eqref{def:solav}. This leads to studying a domain dependent weighted spherical averaging operator \eqref{def:surfavw}.

Instead of resorting to maximal averages and the Bourgain--Stein theorem, an angular decomposition of the operator is carried out. The additional geometric information allows instead to establish good $L^{1}$ bounds that can be interpolated with trivial $L^{\infty}$ bounds in order to obtain a domination of \eqref{def:surfavw} by a converging sum of $L^{p}$ bounded operators. Improving the $L^{1}$ bound over what follows from the behaviour of generic spherical means is crucial when aiming at $L^{p}$ bounds for all $p> 1$. Such a conclusion cannot be drawn from mere polynomial decay of the Fourier transform of the weighted spherical measure in question, if no additional $L^{1}$ information is taken into account. 
Turning the focus from the Littlewood--Paley decomposition and $L^{2}$ methods to an angular decomposition and geometric estimates in $L^{1}$ is the leading insight of the proof. 

The key idea in the $L^1$ estimates can be described as follows.
Each domain $\Omega$ comes endowed with a family of sets  (Figure \ref{fig:PySet})
\[ \{P(y):y \in \Omega\} , \quad P(y) = \{ x \in \Omega: y \in \partial B(x, \dist(x,\Omega^c)) \} ,\]  
which can morally be used to dualize the spherical averaging operators \eqref{def:surfavw} through Fubini's theorem. 
The $L^{1}$ bounds for the constituents in the angular decomposition of the spherical averaging operator correspond to weighted integrals over the pieces of $P(y)$. If $\Omega$ is a ball, then the sets $P(y)$ are ellipsoids with foci at the center of the ball and at $y$. In the cases of the complement of a ball and a half-space, the $P(y)$ take the simple forms of hyperboloids and paraboloids. One cannot hope for as explicit descriptions as that in more general domains, but all $P(y)$ are boundaries of convex sets. This observation is used extensively in the proof. 

\begin{figure}
\centering
\includegraphics[width=0.65\textwidth]{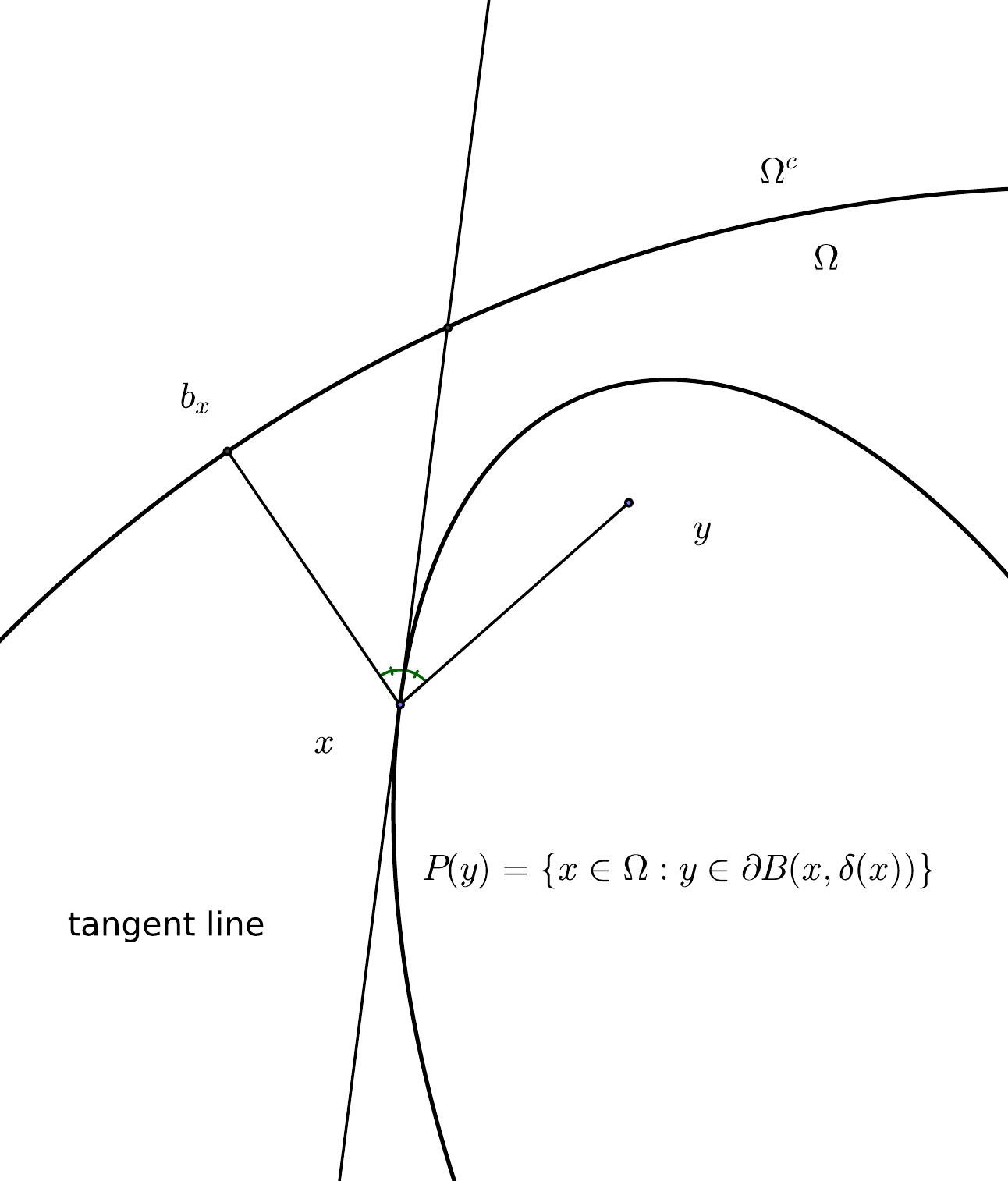}
\begin{caption}
{\label{fig:PySet} A set $P(y)$ and a tangent line.}
\end{caption}
\end{figure}

The structure of the paper is as follows.
In the first section, we introduce notation and some tools that will be helpful throughout the proof. The first  sections are about differentiating the maximal function on so-called unconstrained points and proving the weak differentiability of the maximal function conditionally
to the $L^{p}$ boundedness of the averaging operator \eqref{def:solav}. The rest of the paper is devoted to proving those $L^{p}$ bounds by first computing a formula for the derivative and then carrying out the strategy sketched above. Finally, there is a concluding section with remarks on open problems and certain observations about the proof which might be of independent interest for future research.

\bigskip

\noindent
\textit{Acknowledgement.} 
We would like to thank Juha Kinnunen for turning our attention to the error term in \cite{HKKT2015} and Mart\'i Prats for comments and corrections to the previous version of the paper. Part of the research was done when the second author was visiting the Nonlinear PDE group at Aalto University, which he wishes to thank for its hospitality. JPGR was supported by the Deutscher Akademischer Austauschdienst (DAAD). OS was supported by Hausdorff Center for Mathematics DFG EXC 2047 and the project DFG SFB 1060. JW was supported by the V\"ais\"al\"a Foundation and the Academy of Finland.

\section{Preliminaries}
\subsection{Notation}
We let $n \geq 1$ denote the dimension. For a measurable set $E$, we let $| E |$ denote the $n$-dimensional Lebesgue measure. The $k$-dimensional Hausdorff measures are denoted by $\mathcal{H}^{k}$. An Euclidean ball with center $x \in \mathbb{R}^{n}$ and radius $r > 0$ is denoted by $B(x,r)$. A finite constant only depending on quantities that are not being kept track of is denoted by $C$. If $A \leq C B$ for such constant, we denote $A \lesssim B$ or write $A$ is $\lesssim B$. We write $A \sim B$ if both $A \lesssim B$ and $B \lesssim A$ hold.
 
\subsection{Domains}
\label{sec:dom}
We always assume $\Omega \subset \mathbb{R}^{n}$ to be an open set, which we interchangeably call domain as the distinction obviously plays no role in this paper. We assume it to have non-empty complement. 
The distance function is denoted by $\delta(x) = \dist(x, \Omega^c)$.
As $\Omega^c$ is closed, there exists at least one $b_x \in \Omega^c$ so that $|x-b_x| = \delta(x)$. We reserve the notation $b_x$ for such a point, which need not be unique unless $\Omega^c$ is convex. 
The distance function $\delta: \Omega \to [0, \infty)$ is always $1$-Lipschitz. The gradient exists almost everywhere by Rademacher's theorem, and it holds that 
\begin{equation}\label{eq:graddelta}
\nabla \delta(x) = \frac{x-b_x}{\delta(x)}. 
\end{equation}
This is because clearly the one sided directional derivative of $\delta(x)$ in the direction of $b_x-x$ always exists and is $-1$. Where the gradient exists, we can use $|\nabla\delta(x)|\leq 1$ to conclude that the directional derivative in all directions orthogonal to $x-b_x$ must be zero.

A domain is said to satisfy a uniform curvature bound if there is an $R > 0$ so that for every point $x \in \Omega$ it holds 
\[B\Bigl( b_x +R \frac{x-b_x}{\delta(x)} , R\Bigr) \subset\Omega.\]
All bounded $C^{2}$ domains satisfy this condition, but a domain satisfying a uniform curvature bound might be non-smooth and have inwards-pointing cusps. A domain with an interior ball condition need not satisfy the uniform curvature bound. An example of such a domain is $B(0,1) \setminus ( \{0\} \cup \{(0, \ldots, 0, 2^{k}): k \in \mathbb{Z}, k < 0 \} )$.

%

\subsection{Function spaces on domains}
Functions $f \in L^{p}(\Omega)$ are a priori only defined in the domain $\Omega$, but we always extend them by zero to $\mathbb{R}^{n}$ without additional comments. The Sobolev class $W^{1,p}(\Omega)$ consists of functions $f \in L^{p}(\Omega)$ such that $|\nabla f| \in  L^{p}(\Omega)$. The weak derivatives are defined using test functions in $C^{\infty}_{c}(\Omega)$. 

For the application of the main theorem to the endpoint regularity problem, we need a Sobolev embedding theorem for domains. One concrete case we can deal with is that of a Lipschitz domain.
\begin{proposition}[Section 4.4 in \cite{EG1992}]
Let $\Omega \subset \mathbb{R}^{n}$ be a bounded open set so that $\partial \Omega$ is Lipschitz. Then for every $1 \leq p < \infty$ there exists a bounded extension operator
\[E: W^{1,p}(\Omega) \to W^{1,p}(\mathbb{R}^{n})\]
such that ${\rm supp} (Ef) \subset B(x_0 , 2 \diam(\Omega)) $ for some $x_0 \in \Omega$ and all $f \in W^{1,p}(\Omega)$.
\end{proposition}
By the boundary being Lipschitz, we mean that it can be covered by a finite number of open balls $B_i$ so that for each $i$ the domain $B_i \cap \Omega$ is the epigraph of a Lipschitz function. 

The proposition together with the Gagliardo--Nirenberg--Sobolev inequality (see e.g. Section 4.5.1 in \cite{EG1992}) implies a rudimentary local Sobolev embedding 
\begin{equation}
\label{eq:GNSineq}
\| f \|_{L^{pn/(n-p)}(\Omega) } \leq C_{\Omega,p,n} \|f\|_{W^{1,p}(\Omega)}
\end{equation}
valid for all $f \in W^{1,p}(\Omega)$ whenever $\Omega$ is a bounded open set with Lipschitz boundary. This is sufficient for our purposes.

\subsection{Maximal function}
For $\alpha \in [1,n)$, define the local fractional maximal function relative to $\Omega$ as
\begin{equation*}
\m_{\alpha}f(x) = \sup_{0 < r < \delta(x)} r^{\alpha} \sint_{B(x,r)} f(y) \, dy
\end{equation*}
whenever $f \in L^{1}_{\text{loc}}(\Omega)$. We omit the superscript when $\Omega$ is the whole $\mathbb{R}^{n}$. In addition, we define for $\alpha \in \mathbb{R}$ the auxiliary linear operator
\begin{align}
\label{def:solav}
A_{\alpha} f (x) &  = \delta(x)^{\alpha} \sint_{B(z,\delta(x))} f(y) \, dy  
.
\end{align}

\subsection{Constrained points}
Let $f$ be continuous.
Fix $x \in \Omega$. 
Because the complement of $\Omega$ is non-empty, $\delta(x) < \infty$ and there exists a convergent sequence $r_j \in (0, \delta(x))$ with limit $r = \lim_{j \to \infty } r_j \in [0, \delta(x) ]$ such that  
\[  \m_{\alpha}f(x) = \lim_{j \to \infty}  r_j^{\alpha} \sint_{B(x,r_j)} f(y) \, dy = r^{\alpha} \sint_{B(x,r)} f(y) \, dy\]
if $r>0$.
If 
\[\m_{\alpha}f(x) >  \delta(x)^{\alpha} \sint_{B(x,\delta(x))} f(y) \, dy, \]
the sequence $r_j $ must be chosen so that $r < \delta(x)$, and the point $x$ is said to be \textit{unconstrained}. All other points are called \textit{constrained}. 

\section{The unconstrained part}
The local maximal function behaves similarly to the global one in the unconstrained set, and we reduce the differentiability question of the unconstrained part accordingly to that of the global maximal function. This is the content of the following proposition. 

\begin{proposition}
\label{prop:unconstrained}
Let $p > 1$, $\alpha \geq 1$ and $f \in L^{p}_{\text{loc}}(\Omega)$ be continuous. 
The set $U$ of the unconstrained points is open, the maximal function $\m_\alpha f$ is weakly differentiable in $U$, and the pointwise bound
\[ | \nabla \m_{\alpha} f(x) | \leq c M_{\alpha-1} f (x) \]
holds for a constant $c$ only depending on the dimension and $\alpha$ whenever $x \in U$.
\end{proposition}

\begin{proof}
Consider the fractional average function
\[ A(z,r) :=  r^{\alpha} \sint_{B(z,r)} f(y) \, dy  .\]
It is continuous in $(z,r) \in \Omega \times \R_{+}$. Fix now an unconstrained point $x$. By definition, there exists $\varepsilon > 0$ so that $\m_{\alpha} f(x) - A(x,\delta(x) ) > \epsilon$. Moreover, there exists $\gamma > 0$ so that if $|(z,r) - (x,\delta(x)) | < \gamma$, then $\m_\alpha f(x) - A(z,r) > \varepsilon / 2 $. Since $\m_\alpha f$ is lower semicontinuous, one can find for every $z$ close enough to $x$ a sequence $r_{z,j} \to r_z < \delta(x)- \gamma/2$  so that 
\[\m_\alpha f(z) = \lim_{j \to \infty} r_{z,j} ^{\alpha} \sint_{B(z,r_{z,j})} f(y) \, dy . \]  
In particular, there is an open neighborhood $U_x$ of $x$ so that for all $z \in U_x$ 
\[\m_\alpha f(z) = M_{\alpha}( 1_{B(x,\delta (x) ) } f ) (z). \]
By Theorem 3.1 in \cite{KS2003},
\[ | \nabla \m_{\alpha} f(x) | \leq C M_{\alpha-1} f (x) \]
follows.
\end{proof}

\section{The full maximal function}
Next we prove the differentiability of the local maximal function conditional to $L^{p}$ bounds for the derivative of the averaging operator \eqref{def:solav}. This step 
morally follows from the lattice property of Sobolev functions, but as we only know the weak differentiability of $\m_\alpha f$ in the unconstrained set, some extra work is needed. 

\begin{lemma}
\label{lemma:full}
Let $p>1$, $\alpha \geq 1$ and $\Omega \subset \mathbb{R}^{n}$ be such that $\nabla A_\alpha$ and $M_{\alpha-1}$ are bounded $L^{p}(\Omega) \to L^{p}(\Omega)$. If $f \in L^{p}(\Omega)$, then the local fractional maximal function 
is weakly differentiable and 
\[\| \nabla \m_\alpha  f  \|_{L^{p}(\Omega)} \lesssim \| f \|_{L^{p}(\Omega)}. \]
\end{lemma}

\begin{proof}
Assume first that $f$ is continuous and compactly supported. Following the arguments in \cite{KS2003}, we infer that $\m_\alpha f$ can be seen as supremum over radii between a fixed upper and lower bound. The fractional averages are Lipschitz continuous with constants only depending on the radii, and hence their supremum is also Lipschitz. In particular, we know that $\m_\alpha f$ is continuous.

Denote by $g^{+} = \max (g,0) $ the positive part of a function $g$ and write
\[ \m_\alpha f = (\m_\alpha f - A_\alpha f)^{+} + A_\alpha f. \]
By assumption, the second term admits the desired Sobolev bounds. To deal with the other term, let $\epsilon > 0$ and define 
\[F_\epsilon (t) = \begin{cases}
 ((t - \epsilon)^{2} + \epsilon^{2} )^{1/2} - \epsilon , \quad t > \epsilon \\
 0, \quad t \leq \epsilon .
 \end{cases} \]
These functions are of class $C^{1}(\mathbb{R})$ and converge pointwise to $t \mapsto (t)^{+}$ as $\epsilon \to 0$. Moreover, as $\m_\alpha f$ and $A_\alpha f$ are continuous, $ E = \{ x \in \Omega:  F_\epsilon( \m_\alpha f (x) - A_\alpha f(x) ) > 0 \}$ has its closure contained in the open set of unconstrained points $U$. By Proposition \ref{prop:unconstrained}, the assumption on $A_\alpha$ and the chain rule for Sobolev derivatives (4.2.2 in \cite{EG1992}), we obtain for all partial derivatives $\partial_i$
\[  \partial_i F_\epsilon( \m_\alpha f  - A_\alpha f ) = ( \partial_i \m_\alpha f - \partial_i A_\alpha  ) F_\epsilon' ( \m_\alpha f  - A_\alpha f ) . \]
Taking a test function $\varphi$ and computing
\[\int_{\Omega} F_\epsilon( \m_\alpha f   - A_\alpha f  ) \partial_i \varphi \, dx =   \int_{\Omega}  ( \partial_i \m_\alpha f  - \partial_i A_\alpha f ) F_\epsilon' ( \m_\alpha f  - A_\alpha f)\varphi \, dx , \] 
we see that taking the limit $\epsilon \to 0$ proves the claim for continuous and compactly supported $f$.

To deal with the general case, let $f \in L^{p}(\Omega)$ and let $f_j$ be continuous and compactly supported functions converging to $f$ in $L^{p}$ norm. By $L^{p}$ continuity of the fractional maximal operator, $\m_\alpha f_j \to \m_\alpha f$ in $L^{p}$. As we have proved the following inequality
\[ \| \nabla  \m_\alpha f_j \|_{L^{p}(\Omega)} \lesssim \| f_j \|_{L^{p}(\Omega)}, \]
for continuous functions $f_j$, the sequence $\m_\alpha f_j$ is bounded in $W^{1,p}(\Omega)$. We can extract a weakly convergent subsequence. By taking limits along this sequence and using the uniqueness of distributional limit, we conclude the proof for general $f \in L^p(\Omega)$. 
\end{proof}

As the main theorem is a direct consequence of the previous lemma, it remains to investigate the boundedness of the operator $\nabla A_{\alpha}$ on $L^{p}(\Omega)$.
The following sections are devoted to establishing the required $L^{p}$ bounds when $\Omega$ is sufficiently well-behaved.

\section{Constrained part}
By a change of variables, we can write the averaging operator \eqref{def:solav} as
\[A_\alpha f(x) = \delta(x)^{\alpha} \sint_{B(0,1)} f(x + y \delta(x) ) \, dy.\]
This operator is linear, and as we are aiming for $L^{p}$ bounds, there is no loss of generality in restricting the attention to smooth functions.
If $x$ is a constrained point, then $\m_\alpha f(x) = A_\alpha f(x)$, which justifies our reference to $A_\alpha$ as the constrained part. Also, Lemma \ref{lemma:full} showed that $L^{p}$ bounds for the derivative of $A_\alpha f$ are enough to imply weak differentiability of the full maximal operator, so the maximal function does not play any role in what follows.
A version of the following proposition was already proved in \cite{HKKT2015}, but as we need a formula more precise than what they stated, we include the short proof for clarity.

\begin{proposition}
\label{prop:derivative_averaging}
Let $f \in C^{\infty}(\Omega) $. Then for almost every $x \in \Omega$
\[| \nabla A_\alpha f (x) | \leq c_{n, \alpha}  |A_{\alpha-1}f (x)| +  c_n \delta(x)^{\alpha - 1} \sint_{\partial B(x,\delta(x))} \frac{|y- b_x|}{\delta(x)} f(y) \, d \mathcal{H}^{n-1} (y)    \]
where $b_x \in \partial \Omega$ is a point such that $|b_x - x| = \delta(x)$. 
\end{proposition}

\begin{proof}
Fix a point $x$. As $A_{\alpha}f(x) = \delta(x)^{\alpha} A_0f(x),$ it holds that 
\[\nabla A_{\alpha}f(x) = \alpha \delta(x)^{\alpha-1} A_0f(x) \nabla \delta (x) + \delta(x)^{\alpha}  (\nabla A_0f)(x).\]
Since $|\nabla\delta(x)|\leq1$ (cf.~\eqref{eq:graddelta}), the first summand above is bounded by $A_{\alpha -1}f(x)$. Thus it suffices to analyze the gradient of $A_0f$. Take the unit vector 
\[e=\nabla A_0f(x)/|\nabla A_0f(x)|.\] 
Then
\begin{align*}
|\nabla(A_0 f)(x)|&= ( e \cdot \nabla ) A_0 f (x)\\
	&= \sint_{B(0,1)} \big ( e +  y ( e  \cdot \nabla \delta(x)) \big ) \cdot \nabla f(x+ \delta(x) y) \, dy \\
	&= \frac{1}{\delta(x)} \sint_{B(0,1)} \dive_y \big (  (e +  y ( e  \cdot \nabla \delta(x))) f(x+ \delta(x) y) \big ) \, dy \\
	& \qquad - \frac{n e \cdot \nabla \delta (x)}{\delta(x)} \sint_{B(0,1)} f( x+ \delta(x) y) \, dy 
	 =: \I + \II.
\end{align*}
Since $|\nabla\delta (x)|\leq1$, the contribution $\delta(x)^\alpha\cdot\II$ is pointwise bounded by $n A_{\alpha-1} f$. To estimate the other term, we apply Gauss's theorem to obtain
\begin{align*}
\I &= \frac{c_n}{\delta(x)} \int_{\partial B(0,1)} y \cdot ( e + y (e \cdot \nabla \delta(x)   ) )f (x+ \delta(x) y) \, d \mathcal{H}^{n-1} (y) \\
	& = \frac{c_n}{\delta(x)} \sint_{\partial B(x,\delta (x))} \frac{(y- b_x) \cdot e }{\delta(x)} f(y) dy. 
\end{align*}
So we reach the inequality
\[| \nabla A_\alpha f (x) | \leq \alpha n  |A_{\alpha-1}(x)| + c_n \delta(x)^{\alpha - 1} \sint_{\partial B(x,\delta(x))} \frac{|y- b_x|}{\delta(x)} f(y) \, d \mathcal{H}^{n-1}(y) ,   \]
which proves the claim.
\end{proof}

Because $A_{\alpha-1}f(x)\leq\m_{\alpha-1}f(x)$, and $\m_{\alpha-1}$ satisfies the right $L^{p} \to L^{q}$ bounds, we have reduced the matter to understanding the weighted spherical average
\begin{equation}
\label{def:surfavw}
B_\alpha f(x) := \delta(x)^{\alpha - 1} \sint_{B(x,\delta(x))} \frac{|y- b_x|}{\delta(x)} f(y) \, d \mathcal{H}^{n-1}(y)
\end{equation}
on the right hand side of the conclusion of the previous proposition. The weight $|y- b_x|/\delta(x)$ measures the angle between $b_x - x$ and $y-x$ when $|y-b_x|/\delta(x)$ is small. We decompose the weighted spherical averaging operator according to the angle and location in the domain as follows. For $k\in \mathbb{Z}$, let
\[ \Omega_k = \{ x \in \Omega: 2^{k} \leq \delta(x) < 2^{k+1} \} \]
and for every point $x \in \Omega$ and integer $j \geq 0$ 
\[\omega_j(x) =  \left\lbrace y \in \partial B(x, \delta (x)) : 2^{-j} <  \frac{|y- b_x|}{\delta(x)}  \leq 2^{-j+1} \right \rbrace . \]
Define
\begin{align*}
 S_j^{k} f(x) &=  1_{\Omega_k}(x) \int_{\omega_j(x)}  f(y) \, d \mathcal{H}^{n-1}(y).
\end{align*}
Then
\begin{equation}
\label{eq:red_angle}
B_\alpha f(x) \lesssim \sum_{k \in \mathbb{Z}} \sum_{j=1}^{\infty} 2^{k(\alpha -n) - j }  S_j^{k} f (x)
\end{equation} 
and it remains to prove bounds for $S_j^{k}$ so that the right hand side sums up in $L^p$. This is done by interpolating bounds on $L^{\infty}$ and $L^{1}$. 
\begin{proposition}
\label{prop:infinity}
Let $\Omega$ be any domain. It holds that $\|  S_j^{k} \|_{L^{\infty} \to L^{\infty}} \lesssim 2^{(n-1)(k-j)},$
and consequently $ \| \sum_k  2^{k (1 -n)} S_j^{k} \|_{L^{\infty} \to L^{\infty}} \lesssim  2^{-(n-1)j} $.
\end{proposition}
\begin{proof}
This follows from $\mathcal{H}^{n-1}(\omega_j(x))\lesssim 2^{(n-1)(k-j)}$.
\end{proof}

\section{$L^{1}$ bounds}

To prove $L^{1}$ bounds, we introduce some more notation. For each integer $j \geq 0$ and each point $y \in \Omega$, define 
\begin{equation}
\label{def:pysets}
P_j(y) = \{ x \in \Omega: y \in \omega_j(x) \}, \quad P(y) = \bigcup_{j=0}^{\infty} P_j(y).
\end{equation}
In addition, let 
\begin{equation}
\label{eq:ajk}
A_j^{k} = \bigcup_{x \in \Omega_k} \omega_j(x).
\end{equation}

Formally, certain weighted integrals over $P(y)$ give the adjoint operator of $B_\alpha$. A naive change of order of integration is not justified in this case, but using the decomposition of $B_\alpha$, we can make the idea precise. 
The following two propositions give effective description of $P(y)$ and provide a substitute for Fubini's theorem.
 
\begin{proposition}
\label{prop:convexbody}
Let $\Omega$ be an open set and let $y \in \Omega$. Then
	\[E(y)=\{x\in\Omega:|x-y| \leq \delta(x)\}\]
	is closed and convex set such that
	\[P(y) = \partial E(y) .\]
For each $x \in P(y)$, the supporting hyperplane at $x$ bisects the angle between $y-x$ and $b_x - x$ and is normal to $b_x - y$.
\end{proposition}

\begin{proof}
Recall that $P(y)$ consists of the points with $\{x \in \Omega: |y-x|=\delta(x) \}$. For $x \in P(y)$, it holds that
\[x + \epsilon \frac{b_x -x}{|b_x-x|} \in E(y)^c ,\]
and it is easy to see $\partial E(y) = P(y)$. 
Consider the hyperplane 
\[\{z \in \mathbb{R}^{n} : |z-b_x|=|z-y| \}. \]
It divides the space into two half spaces $H_1 = \{z:|z-b_x|<|z-y|\}$ and $H_2 = \{z:|z-b_x|\geq|z-y|\}$. If $x \in P(y)$, then $E(y) \subset H_2$ and $x \in H_2$. Thus $\partial H_2$ is a supporting hyperplane for $E(y)$ at $x$. As every boundary point of $E(y)$ has a supporting hyperplane, $E(y)$ is convex. The remaining assertions readily follow from the definition of $\partial H_2$. 
\end{proof}

\begin{proposition}
\label{prop:L1bound}
Let $\Omega$ be a domain, $j \geq 0$ and $k$ integers and $f \geq 0$ a bounded continuous function on $\Omega$. Then
	\[\int_{\Omega} S_j^{k}f(x) \, dx \lesssim 2^{j} \sum_{|j'-j|\leq1,|k'-k|\leq1}\int_{A_{j'}^{k'}} f(y) \mathcal{H}^{n-1}(P_{j'}^{k'}(y))\, dy, \]
where we let $P^k_j(y) = P_j(y) \cap \Omega_k.$
\end{proposition}
Note that $y\in A_j^k$ if and only if $P_j^k(y)\neq\emptyset$.

\begin{proof}
The parameter $k$ plays no role in the following computation, but is included in the statement for future reference. 
Let $\varphi \geq 0$ be a smooth function of one variable with compact support in $(0,1)$ and $\|\varphi\|_{L^{1}(\mathbb{R})} = 1$. Denote the $\epsilon$-dilation by $\varphi_\epsilon (t) = \epsilon^{-1} \varphi( t \epsilon^{-1})$. For any fixed $x$, we define the set of relevant directions
\[\omega_j^{\text{dir}}(x)  = \delta(x)^{-1} ( \omega_j(x) - x) \subset \partial B(0,1).\]
As $f$ is positive, the weak convergence 
\begin{align*}
  S_j^{k} f (x)	&=  \int_{\omega_j(x)} f(y) \, d\mathcal{H}^{n-1}(y)  
				=  \lim_{\epsilon \to 0} \int_{ x + \mathbb{R}  \omega_j^{\text{dir}}(x) } f (y) \varphi_{\epsilon}\left( \delta(x) - |x-y| \right) \, dy
\end{align*}
holds.
Integrating over $x$, applying the dominated convergence theorem (this is justified, see the remark at the end of the argument), and using Fubini's theorem, we obtain 
\begin{align}
\label{eq:constr1}
\int_{\Omega_k} S_jf(x) dx 
	\lesssim \int_{ A_j^{k} } f(y)  \left( \lim_{\epsilon \to 0} \frac{ |\{x \in \Omega_k :  y  \in   \omega_j^{\epsilon-}(x)    \}|}{\epsilon}  \right) \, dy 
\end{align}
where the one-sided neighborhood is defined as
	\[\omega_j^{\epsilon-}(x) = x + \omega_j^{\text{dir}}(x) ( \delta(x) - \epsilon, \delta(x) ).  \]

\begin{figure}
\centering
\includegraphics[width=0.8\textwidth]{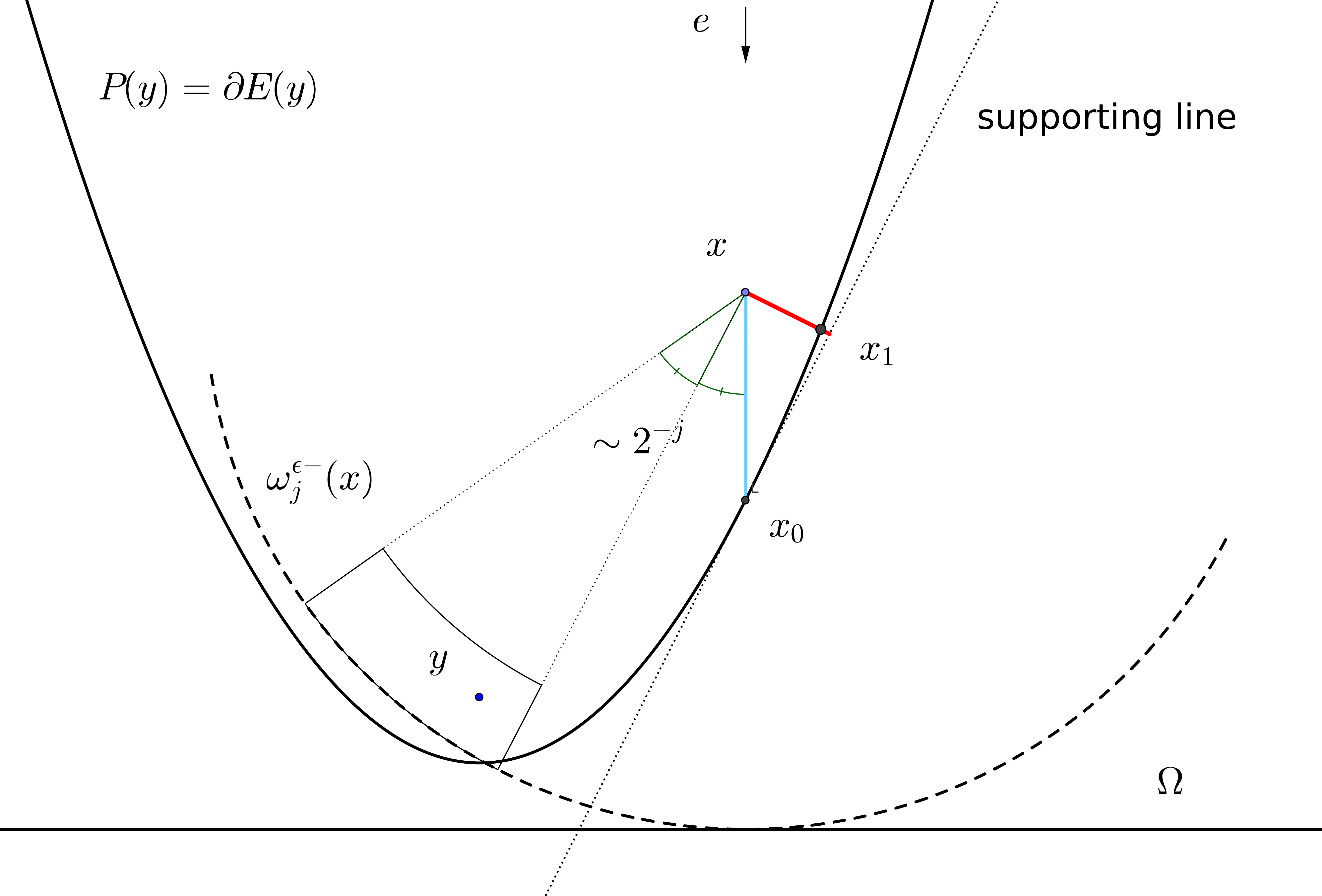}
	\caption{The construction to find $x_1$.}
\end{figure}

Next we estimate the limit expression in \eqref{eq:constr1}. As $j$ and $k$ are fixed, we can assume $\epsilon$ to be very small relative to them.  Let $x \in \Omega_k$. Assume that $y\in\omega_j^{\epsilon-}(x)$. Then
\begin{equation}
\label{eq:yinomegaeps}
-\epsilon<|y-x|-\delta(x)<0
\end{equation}
and by definition $x$ belongs to the interior of $E(y)$.

Set
\[ e = \frac{b_x-x}{|b_x-x|} \]
and let $r\in (0, \delta(x))$ be such that $x + re \in P(y)$.
Next we give an upper bound for $r$.
Because $y\in\omega_j^{\epsilon-}(x)$, it also holds that
\[ \frac{y-x}{|y-x|} \in \omega_j^{\text{dir}}(x).\]
The mapping
\[ g(\rho) := |y-(x+\rho e)|-\delta(x+\rho e)=|y-x-\rho e|-\delta(x)+\rho\]
is Lipschitz and hence absolutely continuous. 

For all $\rho\geq0$ we have the lower bound
\begin{align*}
	g'(\rho) =
		\partial_\rho[|y-(x+\rho e)|-\delta(x+\rho e)]&=-e\cdot\frac{y-x-\rho e}{|y-x-\rho e|}+1\\
		&=1-\cos\measuredangle(b_x-x,y-x-\rho e)\\
		&\geq1-\cos \measuredangle(b_x-x,y-x)\\
		&\gtrsim 2^{-2j}
\end{align*}
The last inequality is due to $y\in\omega_j^{\epsilon-}(x)$. 
Recall that $g(0)\geq-\epsilon$ and $g(r)=0$.
Since $g$ is absolutely continuous, we conclude
\[2^{-2j} r \lesssim \int_{0}^{r} g'(s) \, ds = g(r) - g(0) \leq \epsilon , \]
and
\[r \lesssim 2^{2j} \epsilon.\]

%
%
%
%
%
%

Denote $x_0 = x + r e \in P(y)$.
Consider the $2$-plane containing $x$, $y$, $b_x$ (and $x_0$). Its intersection with the convex body $E(y)$ provided by Proposition \ref{prop:convexbody} is again a convex set $E'$ in the plane. Let $\ell$ be its supporting line at $x_0$. Then $x_0\in P_{j'}^{k'}(y)$ for some $j'\in\{j,j-1\},\ k'\in\{k,k-1\}$ because
\begin{align*}
	\measuredangle(b_x-x_0,y-x_0) &\geq\measuredangle(b_x-x,y-x)\geq2^{-j}\\
	\sin\measuredangle(b_x-x_0,y-x_0) &\leq \frac{|b_x - y|}{\delta(x) - C 2^{2j} \epsilon } =  \frac{|b_x - y|\delta(x)^{-1}}{1 - C \delta(x)^{-1} 2^{2j} \epsilon } \leq \sin2^{-j+2}
\end{align*}
for $\epsilon$ small enough. By Proposition \ref{prop:convexbody} this also means that $y-x_0$ makes an angle $\sim2^{-j}$ with $\ell$, and hence so does $x-x_0$. Let $e'$ be the unit vector perpendicular to $\ell$ and $e' \cdot (y-x) < 0$. 
Then there is 
	\[ s  \lesssim |x-x_0| \sin2^{-j} \lesssim 2^{j} \epsilon \]
so that $x+se' \in \ell$. Since $x\in E'(y)$ and $\ell$ intersects $E'(y)$ only in $\partial E'(y)$, there is $s'<s$ with $x_1 = x+s'e' \in \partial E'(y)$, which means
\begin{equation}
	\dist(x,P(y))\lesssim2^j\epsilon.
	\label{eq:closetoPy}
\end{equation}

Recall that 
\[\dist(x,P_{j'}^{k'}(y))\leq|x-x_0|\lesssim2^{2j}\epsilon.\]
Let 
\[N(\epsilon') = \bigcup_{j'\in\{j-1,j\},k'\in\{k-1,k\}}\{x \in P(y): \dist( x , P_{j'}^{k'}(y)) \leq \epsilon' \}.\]
Then
\begin{multline*}
	\lim_{\epsilon \to 0} \frac{ |\{x \in \Omega_k :  y \in \omega_j^{\epsilon-}(x) \leq \epsilon \}|}{\epsilon}  \\
	\leq \lim_{\epsilon^{'} \to 0 } \lim_{\epsilon \to 0} \frac{ |\{x \in \Omega_k : \dist (x ,P(y)) \leq c_n 2^{j} \epsilon \}\cap N(\epsilon^{'})|}{\epsilon}\\
	 \lesssim \lim_{\epsilon' \to 0} 2^{j} \mathcal{H}^{n-1}(P(y) \cap N(\epsilon'))	\\
	 \leq  \sum_{|j'-j|\leq1,|k'-k|\leq1} 2^{j} \mathcal{H}^{n-1}(P_{j'}^{k'}(y)) ,
\end{multline*}


%
%
where the second inequality follows, for instance, by Theorem 3.2.39 in \cite{Federer}. The integrable majorant of the sequence above that was needed for the application of the dominated convergence theorem before can be obtained by an application of the coarea formula. This completes the proof. 
\end{proof}

These two propositions are enough to conclude a general $L^{1}$ bound for the pieces $S_j^{k}$. This bound can be refined further, when additional regularity on the domain $\Omega$ is assumed.

\begin{proposition}
\label{prop:generaldomain}
	Let $\Omega$ be an open set. Then $\| S_j^{k} \|_{L^{1} \to L^{1}} \lesssim 2^{k(n-1) + j}$. 
\end{proposition}
\begin{proof}
	If $x \in P(y) \cap \Omega_k$, then $|x-y| = \dist(x, \Omega^c) \leq 2^{k+1}$. Hence $P(y) \cap \Omega_k \subset B(y,2^{k+1})$. Recall that $P(y)=\partial E(y)$ and that $E(y)$ is convex. Thus $P(y)\cap\Omega_k\subset\partial(B(y,2^{k+1})\cap E(y))$ where $B(y,2^{k+1})\cap E(y)$ is convex. Since the perimeter of $B(y,2^{k+1})$ dominates the perimeters of all convex sets with non-empty interior contained in it, we can conclude
\begin{multline*}
	\mathcal{H}^{n-1}(P_j^k(y)) \leq \mathcal{H}^{n-1}(P(y) \cap B(y,2^{k+1}) ) \leq \mathcal{H}^{n-1}(\partial(B(y,2^{k+1})\cap E(y))) \\
	\leq \mathcal{H}^{n-1}(\partial B(y,2^{k+1}))  \lesssim 2^{k(n-1)}. 
\end{multline*}
	Now the claim follows from Proposition \ref{prop:L1bound}.
\end{proof}

\begin{remark}
In case $\Omega$ is bounded and $\partial \Omega$ is $C^{2}$ smooth, the estimate for $\mathcal{H}^{n-1}(P_j^{k}(y))$ can be refined as follows. If $x \in P_j^{k}(y)$, then $|y-b_x| \le \delta(x) 2^{-j+1}$. This implies $\dist(y,\partial\Omega) \le \delta(x) \cdot 2^{-j+1}$ and further 
\[
|b_y - b_x| \le |b_y - y| + |y-b_x| \le 4\delta(x)\cdot 2^{-j}.
\]
As the inward-pointing unit normal $N_\Omega$ at the boundary is well-defined and Lipschitz,
\[|N_{\Omega}(b_y) - N_{\Omega}(b_x)| \lesssim \diam(\Omega) 2^{-j}.\]
Because $N(b_z) = (z-b_z)/|z-b_z|$, this implies 
\begin{align*}
\left| N_{\Omega}(b_y) - \frac{(x-y)}{\delta(x)}\right| \le |N_{\Omega}(b_y) - N_{\Omega}(b_x)| + \frac{|y-b_x|}{\delta(x)} \lesssim \diam(\Omega) \cdot 2^{-j}.
\end{align*}
Therefore, all vectors $x-y$ with $y \in \omega_j(x)$ are within an angle $\sim \tilde{c}(\Omega)\cdot 2^{-j}$ of $N_{\Omega}(b_y)$. Hence the set $P_j^k(y)$ is contained in a cylinder of height $\sim 2^k$ and basis $\sim \tilde{c}(\Omega) \cdot 2^{k-j}$. By the inequality for perimeters of convex sets as in the proof of Proposition \ref{prop:generaldomain}
\[
\mathcal{H}^{n-1}(P_j^{k}(y)) \lesssim c(\Omega) 2^k \cdot 2^{(k-j)(n-2)}
.\]
This dependency on $j$ is sharp even for very flat domains as can be seen by letting $\Omega$ be a smoothed out $B(0,10)\cap\{x_1\geq0\}$ and $y=2^{-j}e_1$ and $k\leq0$.

However, as the estimate on $\mathcal{H}^{n-1}(P_j^{k}(y))$ is not the narrow gap of the proof of our main theorem, we do not pursue this aspect further.
\end{remark}

The estimate $\mathcal{H}^{n-1}(P_j^{k}(y)) \lesssim 2^{k(n-1)}$ cannot be improved in general. If the boundary of the domain is a single point, the equality is achieved up to a constant. 
However, focusing on the whole $P_j(y)$ instead of single pieces $P_j^{k}(y)$, one can obtain a different estimate at cost of worsening the dependency on $j$. The following proposition is useful for small values of $j$, and it holds in very general domains.

\begin{proposition}
\label{prop:rough}
Let $\Omega$ be an open set and $y \in \Omega$. Then
\[\int_{ P_j(y) } \frac{1}{\dist(x,y)^{n-1}} \, d \mathcal{H}^{n-1} (x)\lesssim 2^{j } \]
with the constant independent of $y$. In particular, 
\[ \| \sum_k 2^{k(1-n)} S_j^{k} \|_{L^{1} \to L^{1}} \lesssim 2^{2j} \]
\end{proposition}
\begin{proof}
We have
\[\int_{ P_j(y) } \frac{1}{\dist(x,y)^{n-1}} \, d\mathcal{H}^{n-1}(x) 
\lesssim \liminf_{\epsilon \to 0} \frac{1}{\epsilon}  \int_{ \{x \in E(y)^{c} : \dist(x, P_j(y)) \leq \epsilon \} } \frac{1}{\dist(x,y)^{n-1}} \, d x. \]
Given any point $x \in P_j(y)$ and a line $l_x = \{ y + t(x-y): t\in \mathbb{R}\}$, we see that by Proposition \ref{prop:convexbody} the line makes an angle $\sim 2^{-j}$ with $P_j(y)$, and hence
\[\mathcal{H}^{1}(l_x \cap \{z \in E(y)^{c}: \dist(z, P_j(y)) \leq \epsilon \}  ) \lesssim 2^{j}\epsilon .\]
The first claimed bound for the integral follows immediately from passing to polar coordinates with origin at $y$.

To prove the second claim, note that by Proposition \ref{prop:L1bound}
\begin{align*}
 \int_{\Omega} \sum_{k} 2^{k(1-n)} S_j^{k}f(x) \, dx
		&\lesssim 2^{j} \sum_{|j'-j|\leq1}\int_{\Omega} f(y) \left( \sum_{k} 2^{k(1-n)} \mathcal{H}^{n-1}(P_{j'}^{k}(y)) \right) \, dy \\	
		&\lesssim 2^{j} \sum_{|j'-j|\leq1}\int_{\Omega} f(y) \left( \int_{ P_{j'}(y) } \frac{1}{\dist(x,y)^{n-1}} \, d \mathcal{H}^{n-1} (x) \right) \, dy \\	
		& \lesssim 2^{2j}\|f\|_{L^1},
\end{align*}
where the last step was an application of the first claim.
\end{proof}

\section{$L^{p}$ bounds and geometry}
To conclude bounds for the operator $B_\alpha$, we have to sum up all the pieces in the decomposition. In order to make this work, one has to ensure that there is enough decay in $j$ and $k$. Although the $L^{1}$ bounds do not sum up, interpolation with the better $L^{\infty}$ bounds provides us with enough decay in the angle parameter $j$. If $\Omega$ is bounded, we can take advantage of the $L^{p}(\Omega)$ spaces being nested and use the decay in the scale parameter $k$ near the boundary to complete the proof with no smoothness assumptions on the boundary of the domain. This is possible only when we do not attempt to prove scalable estimates that would capture $L^{p} \to L^{q}$ smoothing beyond one derivative gain.

\begin{theorem}
\label{thm:largealpha}
Let $\Omega$ be a bounded open set, $p , \alpha > 1$. Then
\[\| B_\alpha \|_{L^{p}(\Omega) \to L^{p}(\Omega)} \lesssim \diam(\Omega)^{\alpha-1} \] 
where the implicit constant only depends on $p$, $\alpha$ and the dimension.
\end{theorem}
\begin{proof}
Let $S_j = \sum_{k} 2^{k(\alpha - n)} S_j^{k}$  
so that $B_\alpha = \sum_{j} 2^{-j} S_j$. Then by Proposition \ref{prop:generaldomain}
\begin{multline*}
 \| S_j \|_{L^{1}(\Omega) \to L^{1}(\Omega)}  
 \leq \sum_{k=-\infty}^{ \log \diam(\Omega) + 1 }2^{k(\alpha - n)}  \| S_j^{k} \|_{L^{1}(\Omega) \to L^{1}(\Omega)}\\
  \lesssim \sum_{k=-\infty}^{ \log \diam(\Omega)  + 1 }2^{k(\alpha - 1)} 2^{j}
  \lesssim 2^{j}  \diam(\Omega)^{\alpha - 1}.
\end{multline*}
By Proposition \ref{prop:infinity}
\[\| 2^{-j} S_j  \|_{L^{\infty}(\Omega) \to L^{\infty}(\Omega)} \lesssim 2^{-n j } \diam(\Omega)^{\alpha -1} \]
and by interpolation we obtain
	\[  \| 2^{-j} S_j \|_{L^{p}(\Omega) \to L^{p}(\Omega)} \lesssim 2^{- \frac{(p-1) n}p j} \diam(\Omega)^{\alpha-1}. \]
As the exponent is negative, we can sum up in $j$ to conclude the proof.
\end{proof}

To deal with the critical case $\alpha = 1$ where our estimates have the correct scaling, we have to take into account finer properties of the boundary, as the estimation as rough as above leads to a logarithmic blow-up of the $k$-sum at the boundary.

\begin{proposition}
\label{prop:annulus}
Let $\Omega$ be an open set. 
\begin{itemize}
	\item If $\Omega$ satisfies the uniform curvature bound with $R$, then for all $y \in \Omega$ and 
	$x \in P(y)$ with $ \delta(x) \leq R$,   
it holds that
\begin{equation}
\label{eq:annulus_smooth}
	\delta(x) (1 - \frac{\delta(x) }{R}) (1 - \cos \beta ) \leq  \dist (y , \partial \Omega) 
\end{equation}
where $\beta = \measuredangle (b_x-x,y-x)$.  
\item If $\Omega^c$ is convex, then
\begin{equation}
\label{eq:annulus_convex}
\delta(x)(1 - \cos \beta ) \leq \dist (y , \partial \Omega) .
\end{equation}
\end{itemize}
\end{proposition}

\begin{proof}
Take $x \in \Omega$ and $y \in \partial B(x,\delta(x))$ and let $\beta$ be the angle between $b_x - x$ and $y - x$. Because $\Omega$ satisfies a uniform curvature bound, there is an $R > 0$ independent of $x$ and $y$ so that we can find a ball $B(z , R) \subset \Omega$ with  $z=b_x+(x-b_x) R / \delta(x)$ so that $ \overline{B(z,R)} \cap \partial \Omega = \{b_x\}$.
The Pythagorean identity reads
	\begin{align*}
		|z-y|^2&=(\delta(x)\sin\beta)^2+(R-\delta(x)(1-\cos\beta))^2\\
		&=R^2(1-2\frac{\delta(x)}R(1-\frac{\delta(x)}R)(1-\cos\beta))\\
		&\leq R^2(1-\frac{\delta(x)}R(1-\frac{\delta(x)}R)(1-\cos\beta))^2.
	\end{align*}
	Let $w$ be the closest point to $y$ in $\partial B(z,R)$. Since $z$, $y$ and $w$ are on the same line, we get
	\begin{align*}
		\dist(y,\partial\Omega)&\geq|y-w|=|z-w|-|z-y|\\
		&\geq R-R(1-\frac{\delta(x)}R(1-\frac{\delta(x)}R)(1-\cos\beta))\\
		&=\delta(x)(1-\frac{\delta(x)}R)(1-\cos\beta)
	\end{align*}
	as claimed.
If $\Omega^c$ is convex, then the uniform curvature bound is satisfied with $R = \infty$, whence the second claim follows.

\begin{figure}
\centering
\includegraphics[width=0.8\textwidth]{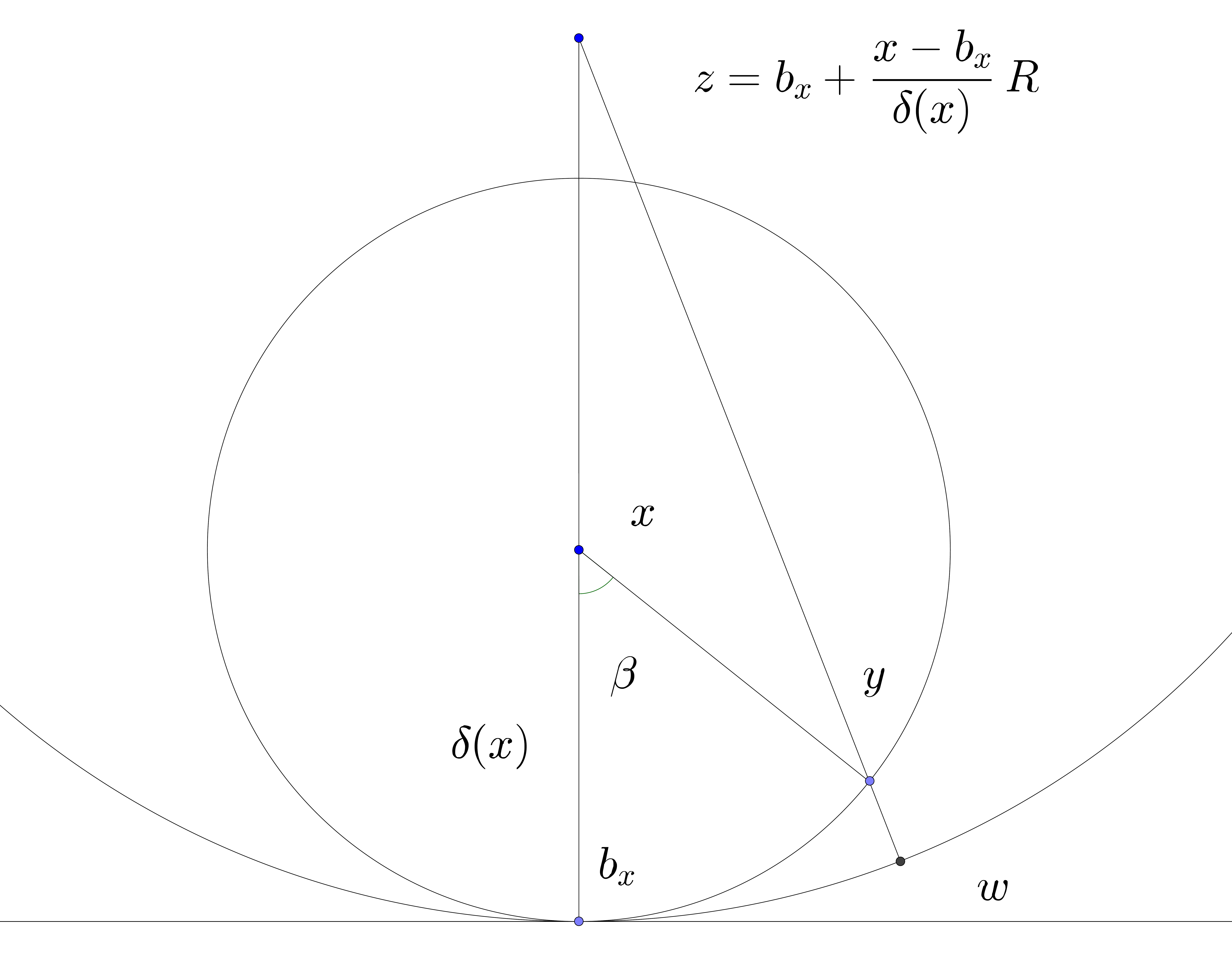}
	\caption{The balls and points appearing in the proof of Proposition \ref{prop:annulus}.}
\end{figure}

\end{proof}

\begin{theorem}
\label{thm:alphaone}
Let $\Omega$ be an open set. Let $\alpha = 1$ and $p > 1$. Then
\begin{itemize}
\item If $ \Omega$ is bounded and satisfies the uniform curvature bound, then 
\[ \| B_1 \|_{L^{p}(\Omega) \to L^{p}(\Omega)} \lesssim \log \left( \frac{\diam(\Omega)}{R} + 1 \right) \]
where $R$ is the radius from the uniform curvature bound.
\item If $\Omega^c$ is convex, then
\[\| B_1 \|_{L^{p}(\Omega) \to L^{p}(\Omega)} \lesssim 1 \]
and the operator norm only depends on the dimension and $p$.
\item If $\Omega$ is merely open, then
\[\| B_1 \|_{L^{p}(\Omega) \to L^{p}(\Omega)} \lesssim 1 \]
under the restriction $p > 1 + \frac{1}{n}$.
\end{itemize}
\end{theorem}
 
\begin{proof}
Proposition \ref{prop:generaldomain} implies
\begin{equation}
\label{eq:thelastequation}
 \int_{\Omega} 2^{k(1-n) - j} S_j^{k} f (x)  \, dx			
			\lesssim   \int f(y) 1_{A_j^k}(y)\, dy.
\end{equation}
Recall the definition \eqref{eq:ajk}. There are only $\sim \log (\diam (\Omega) / R  + 1 )$ values of $k$ so that $R/8 \leq 2^{k} \leq 2 \diam(\Omega)$. For $k$ such that $2^{k+3} \leq R$, 
we can use the first item in Proposition \ref{prop:annulus} to see that for fixed $y$, the set $P_j^k(y)$ is non-empty only for $k$ such that $2^{-2j + k} \lesssim \dist(y, \partial \Omega)$. On the other hand, the upper bound
\[\dist(y, \partial \Omega) \leq |y-b_x| \lesssim 2^{k-j} \]
is always valid, so $P_j^k(y)$ is non-empty only for for $2^{-2j + k} \lesssim \dist(y, \partial \Omega) \lesssim 2^{-j + k}$. Consequently,
\[A_j^{k} \subset \{ y \in \Omega: 2^{-2j + k} \lesssim \dist(y, \partial \Omega) \lesssim 2^{-j + k} \}.  \]
For any $y$, there are only $\lesssim j$ values $k$ such that the set above is non-empty, and hence by \eqref{eq:thelastequation}
%
%
%
%
%
%
%
%
\[ \|\sum_{k}  2^{k(1-n) - j} S_j^{k} \|_{L^{1} \to L^{1}} \lesssim \log \left( \frac{ \diam (\Omega) }{R} +1 \right) + j . \]  
Interpolation as in the proof of Theorem \ref{thm:largealpha} implies the claim.

To prove the second item, just note that the convexity assumption on the complement means sending $R \to \infty$ so that $2^{k+3}\leq R$ always holds. To prove the third item, we study $S_j$ as in the proof of Theorem \ref{thm:largealpha} and replace the $L^{1}$ bound from Proposition \ref{prop:generaldomain} by that from Proposition \ref{prop:rough}. 
\end{proof}

\begin{corollary}
\label{cor:constrained}
Let $\Omega$ be a domain, $p> 1$ and $f \in L^{p}$. Then $A_\alpha f (x)$ from \eqref{def:solav} is weakly differentiable and 
\[\| \nabla A_{\alpha} f  \|_{L^{p}(\Omega)} \lesssim  \| f \|_{L^{p}(\Omega)} \]
if any one of the following holds:
\begin{itemize}
\item $\alpha > 1$ and $\Omega$ is bounded.
\item $\alpha = 1$ and $\Omega$ is bounded and satisfies a uniform curvature bound.
\item $\alpha = 1$ and $\Omega^c$ is convex.
\item $\alpha = 1$ and $p > 1 + \frac{1}{n}$
\end{itemize}
The constant depends on the domain, $\alpha$ and the dimension. 
\end{corollary}
\begin{proof}
By linearity, it suffices to prove the norm inequality for smooth functions.
By Proposition \ref{prop:derivative_averaging}, it suffices to bound $B_\alpha$ from \eqref{def:surfavw}.
This follows from Theorem \ref{thm:largealpha} and Theorem \ref{thm:alphaone}
\end{proof}

Theorem \ref{thm:intro} follows from Corollary \ref{cor:constrained} and Lemma \ref{lemma:full}.

\section{Remarks}

\subsection{Role of the domain}
It is not clear if the conditions on the domain in the hypothesis of Theorem \ref{thm:intro} are necessary. One may ask if
	\[\|\nabla M^\Omega_1\|_{L^p(\Omega)\to L^p(\Omega)}\lesssim 1\]
holds for all domains $\Omega$ and all $p > 1 $. We are not aware of any counterexamples so far.
Since $M_0^\Omega$ does satisfy an $L^p(\Omega)$ bound independent of the domain, the question is about the behaviour of $B_1$ (see Theorem \ref{thm:alphaone}) in general domains. We point out that one avenue for improving the $L^{p}$ bounds for $B_1$ could be to replace the strong $L^{1}$ bounds for $S_j^{k}$ by weak type bounds in order to improve the operator norm bound with respect to $j$.

\subsection{Endpoint regularity in domains}
Corollary \ref{cor:intro2} follows from Theorem \ref{thm:intro}, since
\[
\| \nabla \m_1 f \|_{L^{n/(n-1)}(\Omega)} 
				\lesssim \|f\|_{L^{n/(n-1)}(\Omega)} 
				\lesssim \| f \|_{W^{1,1}(\Omega)}.
\]
Here we used the main theorem and \eqref{eq:GNSineq}. The same observation was done by \cite{CM2015} to notice that the fractional endpoint regularity problem follows from inequality \eqref{eq:KSineq} as $\alpha \geq 1$ in the full space $\mathbb{R}^{n}$. The domain case was not known before as the inequality \eqref{eq:KSineq} should have been replaced by \eqref{eq:introHKKT}. This amounts to changing the Hardy--Littlewood maximal function to the spherical maximal function in the display above. That one is not bounded in $L^{n/(n-1)}$, so the argument breaks down. However, using Theorem \ref{thm:intro}, we can complete the argument in certain domains $\Omega$.

To the best of our knowledge, the fractional endpoint regularity problem has not been studied in domains before. It is hence natural to ask
\begin{question}
What must be assumed about an open set $\Omega \subset \mathbb{R}^{n}$ so that 
\[\| \nabla \m_\alpha f \|_{L^{\alpha /(n- \alpha) }(\Omega)} \lesssim \|f\|_{W^{1,1}(\Omega)} \]
for $\alpha \in (0,1]$?
\end{question}
Our main theorem gives some information on the case $\alpha = 1$, but the remaining values of $\alpha$ remain open. The values $\alpha > 1$ can be dealt with using a spherical maximal function argument with no additional assumptions. The remaining values of $\alpha$ are probably way harder to handle as the endpoint regularity question is completely open even in the full space. 

Finally, we remark that the techniques used to get results for smooth kernels as in \cite{BRS2019} are insensitive to the ambient domain, because one does not use precise information about the maximizing radius. The arguments there only rely on sublinearity of maximal functions. Hence a $W^{1,1}$ variant of Theorem 1.1 in \cite{BRS2019} easily extends to the domain setting. Indeed, fixing $\alpha \in (0,1)$, letting $\Omega$ be any Sobolev extension domain, $\Omega_\epsilon = \{x \in \Omega: \dist(x, \Omega^c) \leq \epsilon \}$ and $m$ a local maximal function with kernel compactly supported and smooth enough as in \cite{BRS2019}, one can invoke Theorem 3 in Section 5.8.2 in \cite{Evans2010} to reduce the problem to proving
\[ \lim_{\epsilon \to 0} \sup_{h \in B(0,\epsilon/2)}  \int_{\Omega_{\epsilon}} \left \lvert \frac{  mf(x+h) - m f(x) }{|h|} \right \rvert ^{\frac{n}{n- \alpha}} \, dx \lesssim \|  f \|_{W^{1,1}(\Omega)} ^{\frac{n}{n- \alpha}} . 
\]
As $f \in W^{1,1}(\Omega)$ coincides with its extension $Ef \in W^{1,1}(\mathbb{R}^{n})$ for all $x \in \Omega$, the integral on the left hand side can be controlled by a maximal multiplier as in \cite{BRS2019} acting on $ Ef( \cdot + h) - Ef(\cdot) $. Then the claim follows from Theorem 3.1 in \cite{BRS2019} and the assumed boundedness of $E : W^{1,1}(\Omega) \to W^{1,1}(\mathbb{R}^{n})$.

\subsection{Smoothing for cube maximal functions}
An equally interesting variant of the local fractional maximal function is the one defined by taking averages over cubes instead of balls
\[M^{\Omega,\text{cube}}_{\alpha}f(x) = \sup_{r > 0, Q(x,r) \subset \Omega} r^{\alpha} \sint_{Q(x,r)} f(y) \, dy  .\]
As the faces of the cubes are completely flat, there are no $L^{p}$ bounds for the maximal function
\begin{equation}
\label{eq:cubesurf}
\sup_{r > 0} \sint_{\partial Q(x,r)} f(y) \, d \mathcal{H}^{n-1}( y ),
\end{equation}
and this was singled out as the principal reason why the methods in \cite{HKKT2015} do not extend to the case of cubical fractional maximal function. 

Although we avoid the use quantites of the type \eqref{eq:cubesurf}, our proof is also inapplicable to the cubical case. There are two obvious obstructions:
\begin{itemize}
\item Let $\Omega$ be the upper half-plane. Take $\delta> 0$ and define $f$ as the characteristic function of $[-\delta,\delta] \times [0, \delta^{s}]$ for some $s \geq 1$. Varying $s$ and sending $\delta \to 0$, we see that 
\[ \| B_1 f\|_{L^{p}} \lesssim \|f\|_{L^{p}} \]
cannot hold for any $p < \infty$.
\item As a detail in the proof, one can note that the analogues of the sets $P(y)$ from \eqref{def:pysets} defined relative to cubes might have full measure. The role of curvature, or lack of it, manifests in the $2^{j}$ factor in the statement of Proposition \eqref{prop:L1bound}.
\end{itemize}

On the other hand, it seems that the problems with the cubical maximal function are not only a matter of lack of curvature. As the remarks above show, there are domains where averages over flat surfaces cause problems. However, if the geometry of the domain is very special, this kind of phenomena can be ruled out. The following observation gives an example.

\begin{proposition}
Let $\Omega = \{(x,y) \in \mathbb{R}^{2} : x < y \}$. Then
	\[ \| \nabla  M^{\Omega,\text{cube}}_{\alpha} f \|_{L^{p}} \lesssim \|f\|_{L^{p}} \]
for all $f \in L^{p}$.
\end{proposition}
\begin{proof}[Sketch of proof]
The reduction to the cubical analogue of \eqref{eq:red_angle} follows by the lines of the spherical proof. Then it suffices to note that the decomposition in $j$ and $k$ is unnecessary, and an $L^{p}$ bound for $p>1$ follows by Minkowski's inequality and a change of variables. 
\end{proof}
The exact behaviour of the cubical local fractional maximal function in more general domains remains an interesting open problem.

\subsection{Scalable estimates}
The method of the proof of Theorem \ref{thm:intro} forced us to prove $L^{p} \to W^{1,p}$ estimates for the derivative of the fractional maximal function. Such estimates can only hold true in bounded domains or for $\alpha = 1$, and in bounded domains they are weaker than the expected $L^{ p} \to \dot{W}^{1,\frac{np}{p-(\alpha-1)}}$ estimates, only known for $p > n/(n-1)$ by \cite{HKKT2015}. We do not pursue this possible improvement direction here, although we 
believe it to be an interesting open problem.

\bibliography{Reference}

\bibliographystyle{abbrv}

\end{document}